\documentclass[12pt]{amsart}

\usepackage{amsmath,graphicx}

\usepackage{tikz-cd}
\usepackage{mathtools}
\usepackage{amsfonts}
\usepackage{graphicx}
\usepackage{verbatim}
\usepackage{textcomp}\usepackage{amssymb}
\usepackage{cite}
\usepackage{color}
\usepackage[all]{xy}
\usepackage{graphicx}
\usepackage{ulem}
\usepackage{color}
\usepackage{hyperref}
\usepackage{soul}
\usepackage{tikz}
\usepackage[all]{xy}
\usepackage{graphicx}

\hypersetup{
    colorlinks,
    citecolor=black,
    filecolor=black,
    linkcolor=black,
    urlcolor=black
}

\usetikzlibrary{backgrounds}
\newtheorem{theorem}{Theorem}[section]
\newtheorem{lemma}[theorem]{Lemma}

\newtheorem{claim}[theorem]{Claim}
\allowdisplaybreaks

\theoremstyle{definition}
\newtheorem{definition}[theorem]{Definition}
\newtheorem{remark}[theorem]{Remark}
\newtheorem{notation}[theorem]{Notation}

{{\sc Proof of the existence, Theorem~\ref{surface-unique}}.}
{{\sc q.e.d.} \\}
{{\sc Proof of the uniqueness,  Theorem~\ref{surface-unique}}.}
{{\sc q.e.d.} \\}

{{\sc Proof of Lemma~\ref{federeroutermeasure}.}}%
{{\sc q.e.d.} \\}
{{\sc Proof of Lemma~\ref{cd2general}.}}%
{{\sc q.e.d.} \\}

{{\sc Proof of Lemma~\ref{usordv}.}}%
{{\sc q.e.d.} \\}
{{\sc Proof of Corollary~\ref{rigiditytheorem}.}}%
{{\sc q.e.d.} \\}
{{\sc Proof of Theorem~\ref{surf}.}}%
{{\sc q.e.d.} \\}
{{\sc Proof of Theorem~\ref{hmco2}.}}%
{{\sc q.e.d.} \\}
{{\sc Proof of claim~(\ref{rickclaim}).}}%
{{\sc q.e.d.} \\}

{{\sc Proof of Theorem~(\ref{RegularityTheoremModelSpace}).}}%
{{\sc q.e.d.} \\}
\numberwithin{equation}{section}

\newtheorem*{theorem*}{theorem}
{{\sc Proof of Lemma~\ref{tri1}.}}%
{{\qed} \\}
{{\sc Proof of Theorem~\ref{regularity}.}}%
{{\qed} \\}
\newenvironment{proof:main}%
{{\sc Proof of Theorem~\ref{theorem:mainexistence}.}}%
{{\qed} \\}

\newenvironment{proof:pluriharmonic}%
    {{\sc Proof of Theorem~\ref{theorem:pluriharmonic}.}}%
  {{\qed} \\}  
  
  \newenvironment{proofof(iii)}%
    {{\sc Proof of $(iii)$.}}%
  {{\qed} \\}  
\newcommand{\R}{\mathbb R}
\newcommand{\domain}{M}

\newcommand{\N}{\mathbb N}

\title[Uniqueness of equivariant harmonic maps]{Uniqueness of equivariant harmonic maps to symmetric spaces and buildings}

\author[Daskalopoulos]{Georgios Daskalopoulos}
\address{Department of Mathematics \\
                 Brown Univeristy \\
                 Providence, RI}%02912}
\email{daskal@math.brown.edu}
\author[Mese]{Chikako Mese}
\address{Johns Hopkins University\\
Department of Mathematics\\
Baltimore, MD}%  21218}
\email{cmese@math.jhu.edu}

\begin{document}

\thanks{
GD supported in part by NSF DMS-2105226, CM supported in part by NSF DMS-2005406.}

\maketitle

\begin{abstract}
We prove  uniqueness of equivariant harmonic maps into  irreducible symmetric spaces of non-compact type and  Euclidean buildings associated to isometric actions by Zariski dense subgroups.  
\end{abstract}

\section{Introduction}

Assume that $M$ and $N$ are  Riemannian manifolds, $M$ has finite volume and  $N$ has non-positive sectional curvature. Hartman \cite{hartman} proved  the following uniqueness result for harmonic maps:  {\it Let $u:M \rightarrow N$ be a finite energy harmonic map of  rank greater at 1 at some point $p \in M$.  If   $N$ has negative sectional curvature at $u(p)$, then $u$ is the only harmonic map in its homotopy class} (cf.~\cite[Corollary following (H)]{hartman}). The second author \cite{mese} generalized Hartman's uniqueness result  to the case when the target space is a geodesic metric space $\tilde X$ with curvature $<0$ in the sense of Alexandrov.  On the other hand, if there exists a 2-plane in $T_{u(p)}N$ with sectional curvature  0 for all $p \in M$, then uniqueness  fails.  For example in the extreme case, when   $N$ is a flat torus, then there exists a family of harmonic maps obtained by translations of a given harmonic map.

Analogous uniqueness statements hold for equivariant harmonic maps.   More precisely, let $\rho:\pi_1(M) \rightarrow \mathsf{Isom}(\tilde X)$ be a homomorphism into the isometry group of an NPC space $\tilde X$ and $\tilde f$ be a $\rho$-equivariant map (cf.~Definition~\ref{def:equivariant}).
Using the same principle as in the homotopy problem,  a finite energy $\rho$-equivariant harmonic map $\tilde u:\tilde \domain \rightarrow \tilde X$ is unique provided $\tilde u$ has rank greater than 1 at some point $p \in \tilde \domain$ and $\tilde X$ has negative  curvature at $\tilde u(p)$.  

In this note, we study uniqueness for equivariant harmonic maps  into irreducible symmetric spaces of non-compact type  and  Euclidean buildings.   We will assume that  Euclidean buildings are locally finite simplicial complexes.  However, we  conjecture that  a similar uniqueness result holds in the case of 
non-locally finite thick Euclidean buildings (cf.~Remark~\ref{concludingremarks}).  The importance of the latter case is that limits of symmetric spaces of non-compact type  are thick Euclidean buildings (cf.~\cite{kleiner-leeb}). This is important in the study of the compactification of character varieties and higher Teichm\"uller theory.

 Symmetric spaces of non-compact type (resp.~Euclidean buildings)  are  examples of Riemannian manifolds of non-positive sectional curvature (resp. NPC spaces or complete CAT(0) spaces).   Harmonic maps  into Riemannian manifolds of non-positive sectional curvature and NPC spaces have been important in the study of geometric rigidity problems (e.g.~\cite{siu},  \cite{corlette}, \cite{gromov-schoen}, \cite{jost-yau}, \cite{mok-siu-yeung}, \cite{daskal-meseGAFA} among many others).  The uniqueness of harmonic maps into symmetric spaces  (resp.~Euclidean buildings) does not follow from  \cite{hartman} (resp.~\cite{mese}) unless $\tilde X$ has rank 1 (resp.~$\tilde X$ is a $\R$-tree). Indeed,  every point $P$ in a rank $n$ symmetric space $\tilde X$  (resp.~$n$-dimensional Euclidean building)
is contained in a convex, isometric embedding of $\R^n$. The novelty  of this paper is that the uniqueness is proven, not with  the assumption on the  curvature bound  as in \cite{hartman} and \cite{mese}, but with an assumption on the homomorphism $\rho:\pi_1(M) \rightarrow \mathsf{Isom}(\tilde X)$.

The main theorem of this paper is the following:

\begin{theorem}[Existence and Uniqueness] \label{existunique}
Let $M$ be a Riemannian manifold with finite volume, $\tilde X$ be an irreducible  symmetric space of non-compact type,  and $\rho:\pi_1(M) \rightarrow \mathsf{Isom}(\tilde X)$ a homomorphism.  Assume:
\begin{itemize}
\item[(i)]  The subgroup $\rho(\pi_1(M))$ does not fix a point at infinity.  
 \item[(ii)] There exists a finite energy $\rho$-equivariant map $\tilde f:\tilde \domain \rightarrow \tilde X$.
\end{itemize}
Then there exists a unique finite energy $\rho$-equivariant harmonic map $\tilde u: \tilde \domain \rightarrow \tilde X$.

The same conclusion holds if $\tilde X$ is a locally finite Euclidean building
with the additional assumption that the action of $\rho(\pi_1(M))$ does not fix a non-empty closed convex strict subset of $\tilde X$.
\end{theorem}

The existence results for harmonic maps is contained in (e.g.~\cite{labourie}, \cite{donaldson}, \cite{corlette}, \cite{gromov-schoen}, \cite{jost}, \cite{korevaar-schoen2}, \cite{korevaar-schoen3}).  Thus, the goal of this paper is to prove the uniqueness assertion in Theorem~\ref{existunique}.
 
The assumptions on the subgroup $\rho(\pi_1(M))$ in Theorem~\ref{existunique} are related to the notion of Zariski dense.  Indeed, in either the case when $\tilde X$ is a symmetric space of non-compact type or an Euclidean building, if the action of the subgroup $\Gamma$ of $\mathsf{Isom}(\tilde X)$ neither fixes a point at infinity nor a non-empty closed convex strict subset, then $\Gamma$ is Zariski dense (cf.~\cite[Proposition 2.8]{caprace-monod}). The converse also holds if $\tilde X$ is a symmetric spaces of non-compact type and  $\mathsf{rank}(\tilde X) \geq 2$ (cf.~\cite[Theorem 4.1]{kleiner-leeb2}), but   there exist Zariski
dense subgroups that fixes a non-empty closed convex strict subset  if $\mathsf{rank}(\tilde X)=1$ (cf.~\cite[Section 4]{caprace}).

 \begin{remark}
For the case when $\tilde X=G/K$ is a symmetric space,  Theorem~\ref{existunique} may be deduced from the gauge theoretic approach due to Donaldson~\cite{donaldson} and Corlette~\cite{corlette2}.  Indeed,  harmonic maps to symmetric spaces can be thought of as a solution to Hitchin's equations and uniqueness follows along the lines  of \cite[Proposition 2.3]{corlette2}.  The point of this paper is to provide a simple geometric proof of the uniqueness of harmonic maps that works for Euclidean buildings as well. 
\end{remark}

{\it Acknowledgement}.
We would like to thank Alexander Lytchak useful discussions.

\section{Preliminaries}

We start with some definitions.  We will assume that $\tilde X$ is a complete metric space.

\begin{definition} A {\it geodesic} $\sigma: I \rightarrow \tilde X$ is a map from an interval $I \subset \R$ such that $d(\sigma(s), \sigma(s+t))= |t|$ for all $s, t \in I$.
  A {\it geodesic line}, {\it geodesic ray} and {\it geodesic segment} are  geodesics with domain $\R$, $[0,\infty)$ and closed interval $[a,b]$ respectively. 
\end{definition}
 
 \begin{definition} 
Geodesics $\sigma:I \rightarrow \tilde X$ and  $\hat \sigma: I \rightarrow \tilde X$ are said to be parallel if there exists a constant $C>0$ such that
\[
d(\sigma(s), \hat \sigma(s)) = C, \ \forall s \in I.
\]
\end{definition}

\begin{remark}
A geodesic  rays $\sigma:[0,\infty) \rightarrow \tilde X$ and $\bar \sigma:[0,\infty) \rightarrow \tilde X$ are asymptotic if there exists a constant $C>0$ such that
\[
d(\sigma(s), \bar \sigma(s)) \leq C, \ \forall s \in\R \ \ (\mbox{resp.}~\forall s \in [0,\infty)).
\]
By \cite[II.2.13]{bridson-haefliger}, the terms {\it parallel geodesic rays} and {\it asymptotic geodesic rays} are equivalent.  
\end{remark}

\begin{definition}
A {\it point at infinity} is an asymptotic class of geodesic rays.  We denote by $[\sigma]$  the asymptotic class containing the geodesic ray  $\sigma$.
\end{definition}

\begin{definition} \label{def:symmetricspace}
A {\it symmetric space} $\tilde X$ is a  Riemannian manifold such that, for any $P \in \tilde X$, there exists $S_P \in \mathsf{Isom}(\tilde X)$ such that $P$ is an isolated fixed point
of $S_P$ and $S_P \circ S_P$ is the identity map. The isometry $S_P$ is called an  {\it inversion symmetry} at $P$.  
\end{definition}

\begin{definition} \label{transvection}
Given a geodesic line $\sigma:\R \rightarrow \tilde X$ and $s \in \R$, the composition
\[
T_s= S_{\sigma(\frac{s}{2})} \circ S_{\sigma(0)}
\]
is called a {\it transvection} along $\sigma$.  
We have that
 \[
T_{s+s'}=T_s \circ T_{s'}
\]
and $\{T_s\}$ forms a one-parameter subgroup of $\mathsf{Isom}(\tilde X)$ that act  as parallel transports along $\sigma$ (cf.~\cite[2.1.1]{eberlein}).
\end{definition}

\begin{definition}[cf.~\cite{bridson-haefliger} Definition 10A.1] \label{def:building}
A {\it Euclidean building of dimension $n$} is a piecewise Euclidean simplicial complex $\tilde X$ such that:
\begin{itemize}
\item[(1)] $\tilde X$ is the union of a collection ${\mathcal A}$ of subcomplexes $A$, called {\it apartments}, such that the intrinsic metric $d_A$ on $A$ makes $(A, d_A)$ isometric to the Euclidean space $\R^n$ and induces the given Euclidean metric on each simplex.  
\item[(2)] Any two simplices $B$ and $B'$ of $X$ are contained in at least one apartment.
\item[(3)] Given two apartments $A$ and $A'$ containing both simplices $B$ and $B'$, there is a simplicial isometry from $(A,d_A)$ to $(A',d_{A'})$ which leaves both $B$ and $B'$ pointwise fixed.
\end{itemize}
Furthermore, will assume
\begin{itemize}
\item[(4)] $\tilde X$ is locally finite.
\end{itemize}
\end{definition}

 \begin{definition}
 A symmetric space of non-compact type $\tilde X$ (resp.~a Euclidean building) is said to be  {\it irreducible}  if it is not isometric to a non-trivial product $\tilde X_1 \times \tilde X_2$ of two symmetric spaces of non-compact type (resp.~Euclidean buildings).  
 \end{definition}
 
 \begin{notation} \label{interpolationnotation}
Given $P, Q \in \tilde X$ and $s \in \R$, we denote 
\[
(1-s)P+sQ
\]
to be the geodesic interpolation between $P$ and $Q$; i.e.~$(1-s)P+sQ=\bar \sigma\left(\delta s\right)$ where $\delta=d(P,Q)$ and $\bar \sigma:[0,\delta]\rightarrow \tilde X$ is a geodesic segment with $\bar \sigma(0)=P$ and $\bar \sigma(\delta)=Q$.
\end{notation}

\begin{definition} \label{def:equivariant}
Let  $\mathsf{Isom}(\tilde X)$ be the group of isometries of $\tilde X$ and   $\rho:\pi_1(M) \rightarrow \mathsf{Isom}(\tilde X)$ be a homomorphism from the fundamental group of a Riemannian manifold $M$. Let $\pi_1(M)$ act on the universal cover $\tilde \domain$ of $M$ by deck transformations.  A map $\tilde f:\tilde \domain \rightarrow  \tilde X$ is said to be $\rho$-equivariant if 
\[
\tilde f(\gamma p) = \rho(\gamma) \tilde f(p), \ \ \forall \gamma \in \pi_1(M), \ p \in \tilde \domain
\]
where we write $gP$ for $g \in \mathsf{Is}(\tilde X)$ and $P \in \tilde X$ instead of $g(P)$ for simplicity.
\end{definition}

If $\tilde X$ is a Riemannian manifold, then  $|d\tilde f|^2$ is the norm of the  differential $d\tilde f:T\tilde \domain \rightarrow T\tilde X$.
If $\tilde X$ is a NPC space, then $|d\tilde f|^2$ is the  energy density function in the sense of \cite{korevaar-schoen}.  Either way, if $\tilde f$ is $\rho$-equivariant, then $|d\tilde f|^2$
 is invariant under the action of $\rho(\gamma)$ for any $\gamma \in \pi_1(M)$, and the energy of $\tilde f$ is defined to be
\[
E^{\tilde f}= \int_M |d\tilde f|^2 d\mbox{vol}_M.
\]

\section{Proof  of Theorem~\ref{existunique}}
\label{sec3}

The existence results for harmonic maps is contained in (e.g.~\cite{labourie}, \cite{donaldson}, \cite{corlette}, \cite{gromov-schoen}, \cite{jost}, \cite{korevaar-schoen2}, \cite{korevaar-schoen3}).
Thus, we need to only prove the uniqueness assertion.

\subsection{Geodesic interpolation} \label{sec:gi}
We assume on the contrary that there exist two distinct $\rho$-equivariant harmonic maps
\[
\tilde u_0: \tilde \domain \rightarrow \tilde X \ \mbox{ and } \tilde u_1: \tilde \domain \rightarrow \tilde X.
\] 
Using Notation~\ref{interpolationnotation}, define the geodesic interpolation of $\tilde u_0$ and $\tilde u_1$; i.e.
 \[
 \tilde u_s:  \tilde \domain \rightarrow \tilde X, \ \ \ \tilde u_s(q)=(1-s)\tilde u_0(q)+s \tilde u_1(q).
  \]
Since $\tilde u_0$ and $\tilde u_1$ are $\rho$-equivariant, $\tilde u_s$ is also $\rho$-equivariant.  
By the  convexity of energy (cf.~\cite[(2.2vi)]{korevaar-schoen}), 
\[
E^{\tilde u_s} \leq (1-s)E^{\tilde u_0} + s E^{\tilde u_1} -s(1-s) \int_M |\nabla d(\tilde u_0, \tilde u_1)|^2 d\mbox{vol}_{\domain}
\] 

\begin{lemma} \label{twothings}
Scaling if necessary, assume 
$d(\tilde u_0(p_0), \tilde u_1(p_0))=1$ for some point $p_0 \in \tilde M$.  Then, for $\tilde u_s$  defined above, we have the following:
\begin{itemize}
\item $d(\tilde u_s(p), \tilde u_1(p)) = s, \  \forall p \in  \tilde M$
 \\
\item  $|(\tilde u_s)_*(V)|^2(p) =  |(\tilde u_0)_*(V)|^2(p)$, for a.e.~$s \in [0,1]$,  $p \in M$,  $V \in T_p\tilde M$. \label{pullback}
\end{itemize}
\end{lemma}

\begin{proof}
Since $\tilde u_0$  and $\tilde u_1$ are energy minimizing,
we conclude
\begin{eqnarray}
0 &= & \int_M |\nabla d(\tilde u_0, \tilde u_1)|^2 d\mbox{vol}_{\domain}
 \label{graddist}
 \\
E^{\tilde u_s}& = & E^{\tilde u_0}, \ \ \forall s \in [0,1] \label{energyconstant}
 \end{eqnarray}

First, (\ref{graddist}) implies that $\nabla d(\tilde u_0, \tilde u_1)=0$~a.e.~in $\tilde \domain$.   Hence, $d(\tilde u_0, \tilde u_1)$  is constant; i.e.
\begin{equation} \label{contantdistancemaps}
d(\tilde u_0, \tilde u_1) \equiv 1.
\end{equation}
For each $q \in \tilde \domain$, define the geodesic segment 
\begin{equation} \label{geoseg}
\bar \sigma_q:[0,1] \rightarrow \tilde X, \ \ \bar\sigma_q(s)=\tilde u_s(q).
\end{equation}
Note that  equality (\ref{energyconstant}) implies that 
\begin{equation} \label{pullback}
|(\tilde u_s)_*(V)|^2(p) = |(\tilde u_0)_*(V)|^2(p), \ \ \mbox{for a.e.~}s \in [0,1], \ p \in \tilde M, V \in T_p\tilde M.
\end{equation}
Indeed,
for $\{P,Q,R,S\} \subset \tilde X$, the quadrilateral comparison for NPC spaces   implies
\[
d^2(P_s,Q_s) \leq (1-s) d^2(P,Q)+sd^2(R,S)
\]
where $P_s=(1-s)P+sS$ and $Q_s=(1-s)Q+sR$.
Applying the above inequality with  $P=\tilde u_0(p)$, $S=\tilde u_1(p)$, $R=\tilde u_1(\exp_p(tV))$ and $Q=\tilde u_0(\exp_p(tV))$ where $t >0$ and $V \in T_p\tilde M$, dividing by $t^2$ and letting $t \rightarrow 0$, we obtain (cf.~\cite[Theorem 1.9.6]{korevaar-schoen})
\[
|(\tilde u_s)_*(V)|^2(p) \leq (1-s) |(\tilde u_0)_*(V)|^2(p) + s|(\tilde u_1)_*(V)|^2(p),
\ 
\mbox{a.e.~$p \in \tilde M, V \in T_p\tilde M$.}
\]   Integrating the above over all unit vectors $V \in T_p\tilde M$  and then over $p \in F$, we obtain
\[
E^{\tilde u_s} \leq (1-s) E^{\tilde u_0}+ s E^{\tilde u_1}.
\]
Combining this with  (\ref{energyconstant}) implies (\ref{pullback}).
\end{proof}

Note that up to this point, we  have only used the fact that $\tilde X$ is an NPC space. 
 We will now specialize to the two cases:  (i) $\tilde X$ is an irreducible symmetric space of non-compact type and (ii) $\tilde X$ is an irreducible Euclidean building.  
%In both cases, 
%we prove that 
%\begin{equation} \label{jw}
%d(\tilde u_0(x), \tilde u_0(y))= d(\tilde u_1(x), \tilde u_1(y)), \ \ \forall x,y \in \tilde M.
%\end{equation}
%We now consider the two cases separately.

\subsection{Symmetric spaces} \label{symmetric space}
   
Throughout this subsection $\tilde X$ is an irreducible symmetric space of non-compact type.  For each $q \in \tilde M$, extend the geodesic segment $\bar \sigma_q$ of (\ref{geoseg})   to  a geodesic line 
\begin{equation} \label{sigmaq}
\sigma_q:\R \rightarrow \tilde X.
\end{equation}
Let
\[
F: \tilde \domain \times \R \rightarrow \tilde X, \ \ \ F(q,s)=\sigma_q(s).
\] 
Let $\nabla^{F^{-1}}$ be the induced connection  on  the vector bundle 
\begin{equation} \label{bb}
(T^*(M \times \R))^{\otimes k} \otimes F^{-1}T\tilde X \rightarrow \tilde \domain \times \R.
\end{equation}
For each $s \in [0,1]$, let $\nabla^{\tilde u_s^{-1}}$ be the induced connection on the vector bundle
\begin{equation} \label{sb}
(T^*M)^{\otimes k} \otimes \tilde u_s^{-1}T\tilde X \rightarrow \tilde \domain.
\end{equation}   
Use the inclusion $\tilde M \rightarrow \tilde M \times \{s\}$ and the identity 
$
F( \cdot,s)= \tilde u_s(\cdot)$ to 
 identify (\ref{sb}) as a subbundle of (\ref{bb}).
 
 Let $s$, $\partial_s$ denote the standard coordinate and coordinate vector on $\R$.  Let $(E_1, \dots, E_n)$ denote a local orthonormal frame of $\tilde \domain$.  Set
\[
V=dF\left( \partial_s \right), \  \ X_\alpha= dF\left(E_\alpha\right) \mbox{ for } \alpha=1, \dots, n
\]
as  sections of $F^{-1}T\tilde X$.
Applying the usual second variation formula of the energy  (e.g.~\cite{eells-sampson}, \cite{schoen}), we  obtain
\begin{eqnarray*}
\frac{d^2}{ds^2} \Big|_{s=t}E^{\tilde u_s}(r)  & = & 2 \int_{\tilde M} \left. \sum_{\alpha=1}^n \left(  \| \nabla^{F^{-1}}_{E_\alpha} V\|^2 - \left\langle R^{\tilde X} \left(V, X_\alpha \right) V, X_\alpha \right\rangle\right)\right|_{s=t} d\mbox{vol}_{\tilde \domain}
\end{eqnarray*}
for any $t \in [0,1]$ where $R^{\tilde X}$ is the Riemannian curvature operator of $\tilde X$.  
By (\ref{energyconstant}), the left hand side of the above equality is equal to 0.  The integrand on the right hand side is non-positive by the assumption of non-positive curvature.  Thus, we conclude that for any $\alpha=1, \dots, n$,
\begin{eqnarray}
\nabla^{F^{-1}}_{E_\alpha} V &  \equiv & 0  \label{U0}
 \nonumber  \\
  \left\langle R^{\tilde X} \left(V, X_\alpha \right) V, X_\alpha \right\rangle & \equiv & 0.   \label{curv0}
  \end{eqnarray}
From the above, we conclude 
  \begin{equation} \label{for}
 \nabla^{F^{-1}}_{\partial_s} (d\tilde u_s(E_\alpha))= \nabla^{F^{-1}}_{\partial_s} X_\alpha 
=0
\end{equation}
  and 
  \begin{equation} \label{lat}
  \nabla^{\tilde u_s^{-1}}_{E_\beta} \nabla^{F^{-1}}_{\partial_s}=  \nabla^{F^{-1}}_{\partial_s}  \nabla^{\tilde u_s^{-1}}_{E_\beta}, \ \forall \beta=1, \dots, n.
  \end{equation}
  Since $\nabla_{\partial_s} E_\alpha = \nabla_{\partial_s} E_\beta=0$, we have
  \begin{eqnarray*}
\lefteqn{ \left(   \nabla^{F^{-1}}_{\partial_s}  \nabla^{\tilde u_s^{-1}} d\tilde u_s \right) (E_\alpha, E_\beta)  }
\\
& = &   \nabla^{F^{-1}}_{\partial_s} \left( \nabla^{\tilde u_s^{-1}} d\tilde u_s(E_\alpha, E_\beta) \right) 
-
\nabla^{\tilde u_s^{-1}} d\tilde u_s( \nabla_{\partial_s}   {E_\alpha}, E_\beta)
 -
\nabla^{\tilde u_s^{-1}} d\tilde u_s(  {E_\alpha}, \nabla_{\partial_s}  E_\beta)
    \\
    & = &    \nabla^{F^{-1}}_{\partial_s} \left( \nabla^{\tilde u_s^{-1}}_{E_\alpha} (d\tilde u_s( E_\beta) )-  d\tilde u_s(  \nabla_{E_\alpha} E_\beta)\right) 
    \\
 & = &     \nabla^{\tilde u_s^{-1}}_{E_\alpha} \left( \nabla^{F^{-1}}_{\partial_s}  (d\tilde u_s( E_\beta)) \right) -  \nabla^{F^{-1}}_{\partial_s}  d\tilde u_s(  \nabla_{E_\alpha} E_\beta)  \ \ \mbox{(by (\ref{lat}))}
 \\
 & = & 0  \ \ \mbox{(by (\ref{for}))}.
   \end{eqnarray*}
More generally, we can inductively use (\ref{lat}) multiple times  to switch the order of differentiation and apply (\ref{for}) to conclude
\begin{equation} \label{io}
\nabla^{F^{-1}}_{\partial_s} \left( \nabla^{\tilde u_s^{-1}} \cdots \nabla^{\tilde u_s^{-1}} d\tilde u_s\right)=0.
\end{equation}

\begin{claim} \label{ut}
Fix a point $p \in \tilde \domain$ and let $T_s \in \mbox{Isom}(\tilde X)$ be the transvection along $\sigma_p$  as in Definition~\ref{transvection}.  
Then 
\[
\tilde u_s=T_s \tilde u_0,
\ \ \
\forall s \in [0,1].
\] 
\end{claim}

\begin{proof}
For $s \in [0,1]$, define a harmonic map
\[
\tilde v_s:  \tilde M \rightarrow \tilde X, \ \ \tilde v_s=T_s \tilde u_0.
\]
Define
\[
\Phi:  \tilde M \times [0,1] \rightarrow \tilde X, \ \ \Phi(q,s)=\tilde v_s(q).
\]
Since $T_s$ is a transvection along the geodesic $\sigma_p$,  
\[
F(p,s)=\tilde u_s(p) = \sigma_p(s)= T_s \sigma_p(0) = T_s\tilde u_0(p)=\tilde v_s(p)=\Phi(p,s).
\]
Furthermore, $T_s$ defines a parallel transport along $\bar \sigma_p$, and thus
\begin{equation} \label{forv}
 \nabla_{\partial_s}^{\Phi^{-1}} \big( d\tilde v_s(E_\alpha) \big) = \nabla_{\partial_s}^{\Phi^{-1}} \big(dT_s (d\tilde u_0(E_\alpha))\big) = 0 \mbox{ at }(p,s), \ \forall s \in (0,1).
\end{equation}
By (\ref{for}) and (\ref{forv}),   the vector fields $d\tilde u_s(E_\alpha)$ and $d\tilde v_s(E_\alpha)$ are both parallel along $\sigma_p(s)$.  Since $d\tilde u_0(E_\alpha)=d\tilde v_0(E_\alpha)$ at $p$, we conclude
\[
d\tilde u_s(E_\alpha) = d\tilde v_s(E_\alpha) \mbox{ at } p \in \tilde M, \ \forall s \in  [0,1].
\]
Next, since $T_s$ is an isometry,
\begin{eqnarray*}
\nabla^{v_s^{-1}}_{E_\alpha} \, \big(d\tilde v_s(E_\beta) \big) 
& = & \nabla^{v_s^{-1}}_{E_\alpha} \, \big(dT_s \circ  d\tilde u_0 (E_\beta) \big)
\\
& = &  \nabla^{\tilde X}_{dT_s \circ d\tilde u_0(E_\alpha)} \, \big(dT_s  \circ d\tilde u_0 (E_\beta)\big)
\\
& = & dT_s\left( \nabla^{\tilde X}_{d\tilde u_0(E_\alpha)} d\tilde u_0 (E_\beta) \right)
\\
& = & dT_s\left(  \nabla^{u_0^{-1}}_{E_\alpha} d\tilde u_0 (E_\beta) \right).
\end{eqnarray*}
Thus,
\begin{eqnarray*}
\nabla^{v_s^{-1}} d\tilde v_s(E_\alpha, E_\beta)
& = &\nabla^{v_s^{-1}}_{E_\alpha} \big(d\tilde v_s(E_\beta) \big)-d\tilde v_s\left(\nabla^{\tilde M}_{E_\alpha} E_\beta\right)
\\
& = & dT_s \left(  \nabla^{u_0^{-1}}_{E_\alpha} d\tilde u_0 (E_\beta) \right)
 -dT_s \left(d\tilde u_0(\nabla^{\tilde M}_{E_\alpha} E_\beta)\right).
\end{eqnarray*}
Since $T_s$ defines a parallel transport along $\bar \sigma_p(s)$, both vector fields on the right hand side are parallel along $\sigma_p(s)$.  Thus,
\[
 \nabla_{\partial_s}^{\Phi^{-1}} \left(  \nabla^{v_s^{-1}} d\tilde v_s(E_\alpha, E_\beta) \right) =0.
\]
Continuing inductively, we can prove
\[
 \nabla_{\partial_s}^{\Phi^{-1}} \left( \nabla^{\tilde v_s^{-1}} \cdots \nabla^{\tilde v_s^{-1}} d\tilde v_s \right)=0.
\]
Combined with (\ref{io})   and the fact that $\tilde u_0=\tilde v_0$, we conclude
\[
\nabla^{\tilde v_s^{-1}} \cdots \nabla^{\tilde v_s^{-1}} d\tilde v_s=\nabla^{\tilde u_s^{-1}} \cdots \nabla^{\tilde u_s^{-1}} d\tilde u_s \mbox{ at }p, \ \forall s \in [0,1].
\]
In other words,   $\tilde u_s$ and $\tilde v_s$ agree up to infinitely high order at $p$ which in turn implies that  $\tilde u_s=\tilde v_s=T_s\tilde u_0$ by  \cite[Theorem 1]{sampson}.
\end{proof}

\begin{claim}  \label{parallel!}
Let $p$ be the point fixed in Claim~\ref{ut} and let  the geodesic ray $\sigma_q:[0,\infty) \rightarrow \tilde X$ be the restriction of the geodesic line defined in (\ref{sigmaq}).  Then
\[
d(\sigma_q(s),\sigma_p(s))= \delta_{p,q}, \ \ \forall q \in \tilde \domain, \ s \in [0,\infty)
\]
where $\delta_{p,q} :=d(\sigma_q(0), \sigma_p(0))$. 
  In particular, $\sigma_q$ is the unique geodesic ray parallel to $\sigma_p$ with value at $s=0$ equal to $\tilde u_0(q)$.  
\end{claim}

\begin{proof}
As above, let $T_s$ be the transvection along $\sigma_p$.  By Claim~\ref{ut},   $\tilde u_s(q)=T_s \tilde u_0(q)$ for $s \in [0,1]$.    
  Since 
  \[
  T_s \tilde u_0(q)=(T_\frac{1}{2} \circ T_{s-\frac{1}{2}})  \tilde u_0(q)=T_\frac{1}{2}  \sigma_q(s-\frac{1}{2}),  \ \ \forall s \in [\frac{1}{2}, \frac{3}{2}],
  \]
   the restriction of $s \mapsto T_s  \tilde u_0(q)$  to $[\frac{1}{2},\frac{3}{2}]$ is a geodesic segment.    Using an analogous argument,  we can inductively  show that for any $n \in \N$, the restriction to $[\frac{n}{2}, \frac{n}{2}+\frac{1}{2}]$ of the map $s \mapsto T_s  \tilde u_0(q)=(T_\frac{1}{2} \circ T_{s-\frac{1}{2}} )\tilde u_0(q)$ is a geodesic segment.   
   Thus, we conclude that   $s \mapsto T_s  \tilde u_0(q)$  is a geodesic ray with $T_s\tilde u_0(q)=\sigma_q(s)$ for $s \in [0,1]$.  Since the two geodesic rays $s \mapsto T_s \tilde u_0(q)$ and $s \mapsto \sigma_q(s)$ agree on $[0,1]$, they are the same geodesic ray. Since $T_s$ is an isometry,  
   \[
   d(\sigma_q(s),\sigma_p(s))=d(T_s  \tilde u_0(q), T_s  \tilde u_0(p))=d(\tilde u_0(q),   \tilde u_0(p))=d(\sigma_0(q),\sigma_0(p)).
   \]
\end{proof}

For $Q = \tilde u_0(q)$, let  $\sigma^Q=\sigma_q$.  By Claim~\ref{parallel!}, there exists map from $\tilde u_0(M)$ to a  family of pairwise parallel geodesic lines given by 
\[
Q \mapsto \sigma^Q.
\]
Since  $\sigma^Q=\sigma_q$ and $\sigma^{\rho(\gamma)Q}=\sigma_{\gamma q}$ are  extensions of $\bar \sigma_q$ and $\bar \sigma_{\gamma q}$ and 
\[
\rho(\gamma) \bar \sigma_q(s)=\rho(\gamma) \tilde u_s(q)=\tilde u_s(\gamma q)=\bar \sigma_{\gamma q}(s), \ \forall s \in [0,1],
\]
we have
\[
\rho(\gamma) \sigma^Q = \sigma^{\rho(\gamma)Q}, \ \ \forall \gamma \in \pi_1(M).
\]
Since $\sigma^Q$ and $\sigma^{\rho(\gamma) Q}$ are parallel geodesic rays, we conclude that 
that \[
\rho(\gamma) [\sigma_q]=[\sigma_q], \ \ \ \gamma \in   \pi_1(M).
\]
In other words, $\rho(\pi_1(M))$ fixes a point at infinity, contradicting assumption (i).

\subsection{Euclidean buildings} \label{sec:Euclideanbuildings}

Throughout this subsection, $\tilde X$ is an irreducible locally finite Euclidean building of dimension $n$.
  An open $n$-dimensional simplex of $\tilde X$ will be referred to as a {\it chamber}.  An apartment of $\tilde X$ is a convex isometric embedding of $\R^n$ in $\tilde X$.

Let $\tilde u_s$ be the geodesic interpolation maps defined in \S\ref{sec:gi}.  The  {\it regular set} ${\mathcal R}(\tilde u_s)$ is the set of all points $q \in \tilde \domain$ with the following property:  There exists a neighborhood  ${\mathcal U}_q$  of $q$ such that $\tilde u_s({\mathcal U}_q)$ is contained in an apartment $A_q$  of $\tilde X$.

\begin{theorem}[\cite{gromov-schoen} Theorem 6.4] \label{gs}
The singular set $\mathcal S(\tilde u_s)$, i.e.~the complement of ${\mathcal R}(\tilde u_s)$,  is a closed set of Hausdorff codimension at least 2.  
\end{theorem}
 For each $q \in \tilde \domain$, let ${\mathcal R}^*(\tilde u_s)$ be the set of points in ${\mathcal  R}(\tilde u_s)$ such that
\begin{equation} \label{88}
\exists \mbox{ $\epsilon>0$ and a chamber $C^*$ such that $\tilde u_s(B_q(\epsilon)) \subset \bar C^*$.}
\end{equation}   
After identifying $A \simeq \R^n$,  $\tilde u_s|_{\mathcal U_q}$ is a harmonic map into Euclidean space, and it follows that the set ${\mathcal U}_q \backslash {\mathcal R}^*(\tilde u_s)$ is a closed set of codimension at least 1. 
Thus, ${\mathcal R}^*(u_q)$ is a closed set of codimension 1.

For each $q \in \tilde \domain$, let $\bar \sigma_q(s)=\tilde u_s(q)$ (cf.~(\ref{geoseg})) and denote by $R_q$ the set of all points $s \in [0,1]$ such that 
\begin{equation} \label{19}
\exists \epsilon>0 \mbox{ and a chamber $C$ such that } \bar \sigma_q((s-\epsilon,s+\epsilon)) \subset \bar C.
\end{equation}
  The complement of $R_q$ in $[0,1]$ is a finite set.
Thus, the complement of  $\mathcal R^{**}=\{(q,s):  q \in \mathcal R^*(\tilde u_{s}), s \in R_{q}\}$ in $\tilde M \times [0,1]$ is an closed set  of measure  0.

Fix $(q,s) \in \mathcal R^{**}$ and let $C$, $C^*$ be the chamber as in (\ref{88}), (\ref{19}) respectively.  
 Let    $(x^1,\dots,x^n)$ be local coordinates in a neighborhood of $q$ with  coordinate vector fields  $(\partial_1, \dots, \partial_n)$.   
Let $A$ be the apartment containing chambers $C$ and $C^*$ as in (\ref{88}) and (\ref{19}).   
After isometrically identifying  $A$ with $ \R^n$, let   $\langle \cdot, \cdot \rangle$ be the usual inner product defined on $A \simeq \R^n$.    
 Thus, (\ref{pullback})  with $V=\partial_\alpha$,
 implies at $q_0$
\[
s \mapsto \left\langle \frac{\partial \tilde u_s}{\partial x^\alpha}, \frac{\partial \tilde u_s}{\partial x^\alpha}\right\rangle = \mbox{constant in $(s_0-\epsilon, s_0+\epsilon)$}.
\]
We  can differentiate this twice with respect to $s$  to obtain
\begin{eqnarray*}
0 = 
 \frac{\partial^2}{\partial s^2}\left\langle \frac{\partial \tilde u_s}{\partial x^\alpha}, \frac{\partial \tilde u_s}{\partial x^\alpha}\right\rangle 
 & = & 2 \left\langle \frac{\partial}{\partial x^\alpha} \frac{\partial^2 \tilde u_s}{\partial s^2}, \frac{\partial \tilde u_s}{\partial x^\alpha}\right\rangle
 + 2  \left\langle \frac{\partial}{\partial x^\alpha} \frac{\partial \tilde u_s}{\partial s}, \frac{\partial}{\partial x^\alpha} \frac{\partial \tilde u_s}{\partial s}\right\rangle.
\end{eqnarray*}
Since $\bar \sigma_{q_0}$ is a geodesic,   $\frac{\partial^2 \tilde u_s}{\partial s^2}(q_0)=\bar \sigma_{q_0}''(s)=0$.
Thus, 
\[
\frac{\partial \bar \sigma'_q(s)}{\partial x^\alpha}\Big|_{q=q_0,s=s_0}= \frac{\partial}{\partial x^\alpha} \frac{\partial \tilde u_s}{\partial s}(q_0)=0.
\]
Since the choice of $\alpha \in \{1,\dots,n\}$ is arbitrary and  $\mathcal R^*$ is of full measure in $\tilde \domain \times [0,1]$, we conclude that  geodesic segments
$\bar \sigma_p$ and $\bar \sigma_q$  are parallel for any $p,q \in \tilde M$; i.e.
\begin{equation} \label{parallelforbuildings}
d(\bar \sigma_p(s),\bar \sigma_q(s))
=:\delta_{p,q}, \ \ \forall s \in [0,1]
\end{equation}
where $\delta_{p,q} :=d(\sigma_q(0), \sigma_p(0))=d(\tilde u_0(q), \tilde u_0(p))$.

Next note the following:
\begin{itemize}
\item (Existence of a geodesic extension)  Given a geodesic segment, property (2) of  Definition~\ref{def:building} implies that there exists an apartment  containing its endpoints and hence its image.  We can thus extend  the geodesic segment to a geodesic line
 in this apartment.
\item (Non-uniqueness of a geodesic extension) Unlike symmetric spaces, the geodesic extensions  are not necessarily unique in a Euclidean building. Indeed, there may be many apartments containing the endpoints of a given geodesic segment.
\end{itemize}
Because of the non-uniqueness of geodesic extensions,  the proof for the building case is slightly different from the symmetric space case as we see below.

We define the sets 
\begin{equation} \label{Cn}
C_0, C_1, \dots, C_n
\end{equation}
inductively follows.  
First, let $C_0=\tilde u_0(\tilde M)$, and then let $C_n$ be the union of the images of all geodesic segments connecting points of $C_{n-1}$.  
The $\rho(\pi_1(M))$-invariance of $C_0$ implies the $\rho(\pi_1(M))$-invariance of $C_n$.

To each $Q=\tilde u_0(q) \in C_0$,  we assign a geodesic segment $\bar \sigma^Q=\bar \sigma_q$ (cf.~(\ref{geoseg})).  By (\ref{parallelforbuildings}), $\{\bar \sigma^Q\}_{Q \in C_0}$ is a  family of pairwise parallel of geodesic segments.  Since $\tilde u_s$ is $\rho$-equivariant, the assignment $Q \mapsto \bar \sigma^Q$ is  $\rho(\pi_1(M))$-equivariant; i.e.~$
\rho(\gamma) \bar \sigma^Q = \bar \sigma^{\rho(\gamma)Q}$ for any $Q \in C_0$ and  $\gamma \in \rho(\pi_1(M)).
$

  For $n \in \N$,  we inductively  a $\rho(\pi_1(M))$-equivariant map from $C_n$ to a family of pairwise parallel geodesic segments  as follows:  For any pair of points $Q_0, Q_1 \in C_{n-1}$, apply the Flat Quadrilateral Theorem (cf.~\cite[2.11]{bridson-haefliger}) with vertices $Q_0, Q_1, P_1:=\bar \sigma^{Q_1}(1), P_0:=\bar \sigma^{Q_0}(1)$ to define a one-parameter family of parallel geodesic segments $\bar \sigma^{Q_t}:[0,1] \rightarrow \tilde X$ with initial point $Q_t=(1-t)Q_0+tQ_1$ and terminal point $P_t=(1-t)P_0+tP_1$  (cf.~(\ref{interpolationnotation})).  
   The inductive hypothesis implies that  the map $Q \mapsto \bar \sigma^Q$ defined on $C_n$ is also  $\rho(\pi_1(M))$-equivariant.
The above construction defines a $\rho(\pi_1(M))$-equivariant map  
\[
Q \mapsto \bar \sigma^Q
\]
 from $\tilde X$ 
to  a family of pairwise geodesic  segments.  Indeed,  we are assuming that the action of $\rho(\pi_1(M))$ does not fix a non-empty closed convex strict subset of $\tilde X$.  Thus,
\begin{equation} \label{convexhullisX}
\tilde X= \bigcup_{n=0}^{\infty} C_n
\end{equation}
since the right hand side is the convex hull of $C_0 = \tilde u_0(\tilde M)$ and each $C_n$ is invariant under the action of $\rho(\pi_1(M))$.

\begin{claim} \label{extinBdg}
There exists a $\rho(\pi_1(M))$-equivariant map
\[
Q   \mapsto \ \ \sigma^Q:[0,\infty) \rightarrow \tilde X
\]
from $\tilde X$ into a family of  pairwise parallel rays; i.e.~$\rho(\gamma)\sigma^Q=\sigma^{\rho(\gamma)Q}$ for all $Q \in \tilde X$, $\gamma \in \pi_1(M)$ and $
d(\sigma_p(s), \sigma_q(s))= \delta_{p,q}$ for all $s \in [0,\infty)$. \end{claim}

\begin{proof}
For $Q \in \tilde X$, we inductively construct a sequence $\{Q_i\}$ of points in $\tilde X$ by first setting $Q_0=Q$ and then defining $Q_i= \bar \sigma^{Q_{i-1}}(\frac{3}{4})$.
Next, let 
\[
L^Q=\bigcup_{i=0}^\infty I^{Q_i}
\]
where $I^{Q_i}=\bar \sigma^{Q_i}([0,1])$. Therefore, $L^Q$ is  a  union of pairwise parallel geodesic segments.  Thus,  $\{L^Q\}_{Q \in \tilde X}$ is a family of  pairwise parallel geodesic rays. Moreover,  the $\rho(\pi_1(M))$-equivariance of the map  $Q \mapsto \bar \sigma^Q$ implies 
$\rho(\gamma)\bar \sigma^{Q_{i-1}}(\frac{3}{4})=\bar \sigma^{\rho(\gamma)Q_{i-1}}(\frac{3}{4})$.  Thus, if $\{Q_i\}$ is the sequence constructed starting with $Q_0=Q$, then $\{\rho(\gamma)Q_i\}$ is the sequence constructed starting  with $\rho(\gamma)Q_0= \rho(\gamma)Q$.  We thus conclude 
\[
\rho(\gamma)L^Q=\bigcup_{i=-\infty}^{\infty} \rho(\gamma) I^{Q_i} = \bigcup_{i=-\infty}^{\infty} I^{\rho(\gamma)Q_i}=L^{\rho(\gamma)Q}.
\]  
We are done by letting the geodesic ray $\sigma^Q:[0,\infty) \rightarrow \tilde X$ be the extension of the geodesic segment $\bar \sigma^Q:[0,1] \rightarrow \tilde X$ parameterizing $L^Q$.
\end{proof}

Claim~\ref{extinBdg} implies that $\rho(\pi_1(M))$ fixes the point $[\sigma^Q]$ at infinity. This  contradicts assumption (i) and completes the proof.

\begin{remark}[Generalization to thick Euclidean buildings] \label{concludingremarks}

Conjecturally, a harmonic map into a
thick Euclidean building  has locally finite image and the Gromov-Schoen regularity result holds in this case. If this is the case, then 
Theorem~\ref{existunique} holds  when $\tilde X$ is assumed to be a thick  Euclidean building.  Indeed, the only place where the assumption that the Euclidean building is locally finite is used is in the application of Theorem~\ref{gs}.  (More precisely,  the target space is  assumed to be locally finite in the regularity result of Gromov-Schoen ~\cite[Theorem 6.4]{gromov-schoen}.)
  The rest of the proof of Theorem~\ref{existunique} goes through exactly as in the locally finite Euclidean building case. 
  \end{remark}


\begin{thebibliography}{ABC}


\bibitem[BH]{bridson-haefliger}  M.~R.~Bridson and A.~Haefliger.  {\it Metric Spaces of Non-Positive Curvature}.  Springer-Verlag, Berlin 1999.

\bibitem[BT]{bruhat-tits} F.~Bruhat and  J.~Tits.  {\it Groupes  r\'eductifs  sur  un  corps  local.}   Publ.~Math.~IHES 41 (1972) 5-251.

\bibitem[Ca]{caprace} P.-E.~Caprace. {\it Lectures on proper CAT(0) spaces and their isometry groups.} 
Geometric group theory, IAS/Park City Math. Ser., vol. 21, Amer. Math. Soc., Providence,
RI, 2014, pp. 91-125. 

\bibitem[CaMo]{caprace-monod} P.-E.~Caprace. and N.~Monod. {\it 
Isometry groups of non-positively curved spaces: discrete subgroups}. 
Journal of Topology 2 (2009) 701-746.



\bibitem[Co1]{corlette} K.~Corlette. {\it Archimedian superrigidity and hyperbolic geometry.} Ann. Math. 135 (1990) 165-182.

\bibitem[Co2]{corlette2}  K.~Corlette. {\it Flat $G$-bundles with anonical metrics}.   J.~Differential Geom. 28 (1988) 361-382.


 
\bibitem[DMV]{daskal-meseGAFA} G.~Daskalopoulos, C.~Mese and A.~Vdovina. {\it Superrigidity of Hyperbolic Buildings.} GAFA 21 (2011) 1-15.



\bibitem[Do]{donaldson} S.~Donaldson. {\it Twisted harmonic maps and the self-duality equations.} Proc.~London Math.~Soc.~55 (1987) 127-131.

\bibitem[Eb]{eberlein} P.~Eberlein.  {\it Geometry of Nonpositively Curved Manifolds}. Chicago Lecture Series in Mathematics. Chicago, 1997.
\bibitem[ES]{eells-sampson}
J.~Eells, Jr. and J.~H.~Sampson.  {\it Harmonic Mappings of Riemannian Manifolds.}
American Journal of Mathematics
86 (1964) 109-160.

\bibitem[GS]{gromov-schoen} M. Gromov and R. Schoen. {\it Harmonic maps into singular
spaces and $p$-adic superrigidity for lattices in groups of rank
one.} Publ. Math. IHES 76 (1992)  165-246.



\bibitem[Ha]{hartman} P.~Hartman.  {\it On homotopic harmonic maps.} Canad.~J.~Math.19 (1967) 673–687.

\bibitem[He]{helgason} {\it Differential Geometry, Lie Groups and Symmetric Spaces}, Graduate Studies in Mathematics Vol 34.  American Mathematical Society, Providence, R.I.
\bibitem[J]{jost}  J. Jost.  {\it Nonpositive curvature:
geometric and analytic aspects.}  Lectures in Mathematics.  ETH
Z\"{u}rich, Birkh\"{a}user Verlag 1997.

\bibitem[JY]{jost-yau} J.~Jost and S.~T.~Yau.  {\it  Harmonic maps and superrigidity}.   Differential Geometry: Partial Differential Equations on Manifolds, Proc. Sympos. Pure Math. 54, part 1 (1993) 245-280.

\bibitem[KL1]{kleiner-leeb} B.~Kleiner and B.~Leeb.  {\it 
 Rigidity of quasi-isometries for symmetric spaces and Euclidean buildings.} 
Publications Math\'ematiques de IHES  86 (1997) 115-197.
 
\bibitem[KL2]{kleiner-leeb2} B.~Kleiner and B.~Leeb. {\it Rigidity of invariant convex sets in symmetric spaces.} Invent.~Math. 163 (2006)
657-676.

\bibitem[KS1]{korevaar-schoen} N.~Korevaar and R.~Schoen.  {\it Global existence theorems for harmonic maps to non-locally compact spaces.} Comm.~Anal.~Geom. 5 (1997) 213-266.

\bibitem[KS2]{korevaar-schoen2}  N. Korevaar and R. Schoen.  {\it
Global existence theorem for harmonic maps to non-locally compact spaces.}  Comm.  Anal. Geom. 5 (1997) 333-387.

\bibitem[KS3]{korevaar-schoen3}  N. Korevaar and R. Schoen.  {\it
Global existence theorem for harmonic maps:  finite rank spaces and an approach to rigidity for smooth actions.}  Unpublished manuscript.


\bibitem[L]{labourie} F.~Labourie.  {\it  Existence d’applications harmoniques tordues a valeurs dans lesvarietes a courbure negative.} Proc.~Amer.~Math.~Soc. 111 (1991) 877-882.

\bibitem[Me]{mese} C.~Mese.  {\it Uniqueness theorems for harmonic maps into metric spaces.} Communications in Contemporary Mathematics 4 (2002) 725-750.


\bibitem[MSY]{mok-siu-yeung} N. Mok, Y.-T. Siu, and S.~K. Yeung, {\it Geometric superrigidity} Invent.~Math. 113 (1993) 57–83.




\bibitem[Sa]{sampson} J.~H.~Sampson.  {\it Some properties and applications of harmonic mappings}.  Annales scientifiques de l'\'Ecole Normale Sup\'erieure, Series 4 (1978) 211-228. 
\bibitem[Sc]{schoen} R.~Schoen. {\it Analytic aspects of the harmonic map problem.} Seminar on Non-linear Partial Differential Equations (S.S. Chern, ed.), MSRI Publ. 2, Springer-Verlag, New York 1983, 321–358.

\bibitem[Si]{siu} Y.-T.~Siu. {\it The complex-analyticity of harmonic maps and the strong rigidity of compact Kahler manifolds}  Ann.~of~Math. 112 (1980) 73-112.

\end{thebibliography}
\end{document}